\documentclass[12pt,a4paper]{article}
\usepackage[latin2]{inputenc}
\usepackage{amssymb}
\usepackage{amsmath}
\usepackage{amscd}
\usepackage{amsthm}
\usepackage{enumerate}
\usepackage{array}
\usepackage{tabularx}
\usepackage{t1enc}
\usepackage{amsthm}
\usepackage[mathscr]{eucal}
\usepackage{graphicx}
\numberwithin{equation}{section}
\newtheorem{theorem}{Theorem}[section]

\newtheorem{lemma}[theorem]{Lemma}

\newtheorem{claim}[theorem]{Claim}

\newtheorem{problem}[theorem]{Problem}

\newtheorem{nota}[theorem]{Notation}

\newtheorem{remark}[theorem]{Remark}

\newcommand{\cC}{\mathcal{C}}
\newcommand{\cM}{\mathcal{M}}
\newcommand{\cP}{\mathcal{P}}
\newcommand{\cL}{\mathcal{L}}

\newcommand{\cH}{\mathcal{H}}

\newcommand{\PG}{\mathrm{PG}}

\def\pd{\hbox{\rm pd}}
\def\diam{\hbox{\rm diam}}

\newcommand{\cut}[1]{}

\begin{document}

\title{Partition dimension of projective planes}

\author{{\bf Zoltán Blázsik,  Zolt\'an L\'or\'ant Nagy}\\
{\small MTA--ELTE Geometric and Algebraic Combinatorics Research Group}\\
{\small E\"otv\"os Lor\'and University}\\
{\small H--1117 Budapest, P\'azm\'any P.\ s\'et\'any 1/C, Hungary}\\
{\small \tt{blazsik@cs.elte.hu}, \tt{  nagyzoli@cs.elte.hu}}}

\date{}

\maketitle

\begin{abstract}
	We determine the partition dimension of the incidence graph $G(\Pi_q)$ of the projective plane $\Pi_q$ up to a constant factor $2$ as $(2+o(1))\log_2{q}\leq \pd(G(\Pi_q))\leq (4+o(1))\log_2{q}.$


\end{abstract}

\section{Introduction}

In graph theory, a large number of different concepts were introduced to distinguish  or identify every vertex in a given graph.
Notably the vertices are usually distinguished via adjacency to a certain set of vertices --- like in case of identifying codes in graphs and distinguishing sets or locating-dominating sets \cite{babai, ide1, ide2, slater} --- or via distance from a certain set of vertices --- like in case of resolving sets, metric and partition dimension \cite{BC, Ch}. In this paper we will study the latter concept on a highly symmetric graph family of great importance in many branches of combinatorics, namely on the incidence graphs of  projective planes.

For vertices $u$ and $v$ in a connected graph $G$, the distance $d(u, v)$ is the length of a
shortest path between $u$ and $v$ in $G$.
For an ordered set $W = \{w_1, w_2, \ldots , w_k\}$ of
vertices in a connected graph $G$ and a vertex $v$ of $G$, the $k$-vector (ordered $k$-tuple)
$$r(v|W)=(d(v, w_1), d(v, w_2), \ldots, d(v, w_k))$$
is referred to as the\textit{ (metric) representation} of $v$ with respect to $W$. The set $W$
is called a \textit{resolving set} for $G$ if the vertices of $G$ have distinct representations.
A resolving set containing a minimum number of vertices is called a minimum
resolving set of $G$. The number of vertices is the so called
\textit{metric dimension} and denoted by $\mu(G)$.\\
The distance concept was naturally generalized to subsets of points due to Chartrand et al. \cite{Ch}.\\
For an ordered $k$-partition $\textbf{S} = \{S_1, S_2, \ldots , S_k\}$ of $V(G)$ and a vertex $v$ of $G$, the
representation of $v$ with respect to  $\mathbf{S}$ is defined as the $k$-vector
$$r(v|\textbf{S}) = (d(v, S_1), d(v, S_2),\ldots , d(v, S_k)).$$

 The partition $\textbf{S}$ is called a \textit{ resolving partition} if the $k$-vectors $r(v | \textbf{S})$, $v \in V(G)$,
are distinct. The minimum $k$ for which there is a resolving $k$-partition of $V(G)$ is
the \textit{partition dimension} $\pd(G)$ of $G$. We note that the notation also make sense if $\mathbf{S}$ is any family of sets.
 Generally, we say that a set system \textit{separates} a vertex set if no two vertices have equal distances from every set of the set system.


Both dimension concepts has been widely investigated, see \cite{BC, CGH, T1} for surveys.
Although they are analogously defined and there are connections between the two parameters, in general, they are not similar in nature.
To illustrate this phenomenon, we recall some results concerning $\mu(G)$ and $\pd(G)$.

\begin{claim}\cite{Ch} $\pd(G)\leq \mu(G)+1$ for all graphs $G$.
\end{claim}

\begin{claim}\cite{CGH} Given natural numbers $\alpha$ and $\beta$ where $3 \leq \alpha \leq \beta + 1$, there exists a graph $G$ where
$\mu(G) = \beta$ and $\pd(G) = \alpha$.
\end{claim}

The study of dimension parameters concerning incidence graphs of designs or geometries has been initiated only recently in  \cite{Bailey, Bailey1, CGH, He}.  Note that a similar concept of identifying codes in special graphs is also recently studied \cite{ide0, ide2}.

Let $\Pi_q$ be an arbitrary finite projective plane of order $q$ of a point set $\cP$ and a line set $\cL$. We denote the plane by $\PG(2,q)$ if we assume that the plane is built on a finite field $\mathbb{F}_q$.
The incidence graph of a plane $\Pi_q$ is denoted by $G(\Pi_q)$.  We will denote the classes of the bipartite graph $G(\Pi_q)$ by $\cP[G(\Pi_q)]$  and $\cL[G(\Pi_q)]$, corresponding to points and lines, respectively. Similarly, we introduce this notation for any subset $Z$ of the vertex set of $G(\Pi_q)$ in general, namely $\cP[Z])$ will denote those vertices of $Z$ which correspond to the points of $\Pi_q$ and $\cL[Z]$ will denote those vertices of $Z$ which correspond to the lines of $\Pi_q$.

The metric dimension of this incidence graph was determined by   H\'eger and Tak\'ats in \cite{He}.

\begin{theorem}\label{marcella} If $q$ is large enough, then $\mu(G(\Pi_q))=4q-4$.
\end{theorem}

Chapell, Gimbel and Hartman \cite{CGH} gave bounds on the partition dimension $\pd(G)$ in terms of the \textit{diameter} $\diam(G)$ of the graph $G$ and investigated the case $\diam(G)=2$. They mentioned that investigating the order of a graph with given partition dimension and diameter appears to
be more difficult when the diameter exceeds two. The incidence graph of projective planes provides an infinite family for well structured graphs of diameter three, so this can be considered as a partial motivation as well for the following problem, besides H\'eger and Tak\'ats' result Theorem \ref{marcella}.

\begin{problem} Determine the partition dimension of the incidence graph of a finite projective plane.
\end{problem}

Our main results are as follows.

\begin{theorem}\label{also} The partition dimension of the incidence graph of a projective plane of order $q$ is at least of order $(2+o(1))\log_2{q} $.
\end{theorem}

\begin{theorem}\label{felso} The partition dimension of the incidence graph of a projective plane of order $q$ is at most of order $(4+o(1))\log_2{q} $.
\end{theorem}

Note that in view of a general bound of Theorem 3.1 in \cite{CGH} concerning the maximal degree of the graph, $\pd(G(\Pi_q))\geq \log_3 (q+2)$.
\medskip

Our paper is built up as follows.
In Section 2. we prove Theorem \ref{also}, then Section 3. is devoted to derive Theorem \ref{felso} using probabilistic and graph theoretic tools. A survey on applications of the probabilistic method in finite geometry can be found in \cite{gacs}. Finally we discuss  open problems in Section 4.

\section{Lower Bound}

We are going to show that $\pd(G(\Pi_q))$ is at least of size $(2+o(1))\log_2{q} $. To this end, let's consider a resolving partition $\textbf{S}$ with sets $\{\cP_1, \cP_2, \dots, \cP_r, \cL_1, \cL_2, \dots, \cL_s, \cM_1, \cM_2, \dots, \cM_t \}$ where $\cP_i \subseteq \cP[G(\Pi_q)]$, $L_j \subseteq \cL[G(\Pi_q)]$ and $\cM_k$ is a mixed subset containing vertices from both $\cP[G(\Pi_q)]$ and $\cL[G(\Pi_q)]$. Since $\textbf{S}$ is a resolving partition we know that for all vertices of the incidence graph the corresponding vectors are pairwise different. Let's examine the possible values of the coordinates of these vectors depending on the type (corresponding to point or line) of the vertex $v$.



\begin{enumerate}
\item If $v \in \cP[G(\Pi_q)]$ then
\begin{itemize}
\item $d(v,\cP_i) = \left \{ \begin{array}{ll} 0 & v \in \cP_i \\ 2 & \textrm{otherwise}. \end{array} \right .$
\item $d(v,\cL_j) = \left \{ \begin{array}{ll} 1 & \textrm{there is a line in } \cL_j \textrm{ which is incident with } v \\ 3 & \textrm{otherwise}. \end{array} \right .$
\item $d(v, \cM_k) = \left \{ \begin{array}{ll} 0 & v \in \cM_k \\ 1 & \textrm{there is a line in } M_k \textrm{ which is incident with } v \\ 2 & \textrm{otherwise}. \end{array} \right .$
\end{itemize}
\item If $v \in \cL[G(\Pi_q)]$ then
\begin{itemize}
\item $d(v,\cL_j) = \left \{ \begin{array}{ll} 0 & v \in \cL_j \\ 2 & \textrm{otherwise}. \end{array} \right .$
\item $d(v, \cP_i) = \left \{ \begin{array}{ll} 1 & \textrm{there is a point in } \cP_i \textrm{ which is incident with } v \\ 3 & \textrm{otherwise}. \end{array} \right .$
\item $d(v, \cM_k) = \left \{ \begin{array}{ll} 0 & v \in \cM_k \\ 1 & \textrm{there is a point in } M_k \textrm{ which is incident with } v \\ 2 & \textrm{otherwise}. \end{array} \right .$
\end{itemize}
\end{enumerate}

Note that if there is a partition class which contains only one type of vertices then by this last observation one can distinguish all the vertices of $\cP[G(\Pi_q)]$ from the vertices of $\cL[G(\Pi_q)]$. Moreover if we consider two vertices which are of the same type then their vector is different if they are not in the same partition class. But we know that there is a partition class which contains at least $\frac{q^2+q+1}{s+t}$ many lines and also there is a class which contains at least $\frac{q^2+q+1}{r+t}$ many points.

For these lines, their representation vectors are the same in the coordinates corresponding to the subsets $\cL_j$. The number of coordinates remaining to distinguish two such lines depends on the type of their partition class. Namely if this class is a class with just lines then it's $r+t$ but when this class is a mixed class then it's just only $r+t-1$. Since the values of these remaining coordinates could only be 1 or 3 for the coordinates corresponding to a subset $\cP_i$ and could only be 1 or 2 for the mixed classes because they do not contain these. Hence the following inequality has to hold:

\[
2^{r+t-1} \ge \frac{q^2+q+1}{s+t}.
\]

A similar argument works for the points too and it gives that

\[
2^{s+t-1} \ge \frac{q^2+q+1}{r+t}.
\]

We would like to minimize $r+s+t$ under the above conditions. One can easily see that the minimum could be reached by taking $s=r=0$ (Note: in this setup one should be careful because the points and the lines are not automatically distinguished). In that case the two condition happens to be the same:

\[
t \cdot 2^{t-1} \ge q^2+q+1.
\]

Hence the theoretical lower bound follows on the partition dimension number.

\section{Upper bound}



Here we prove that $\pd(G(\Pi_q))\leq (4+o(1)) \log_2 q$. First we will outline the construction that provides the desired bound and introduce some key tools. Next we prove two lemmas concerning main ingredients of the constructions. Finally we show that the construction is indeed a resolving partition.

\begin{nota} \normalfont
Let us choose an incident pair of  point $\widetilde{P_0}$ and line $\widetilde{\ell_0}$ and we call them the \textit{support} of the construction. Denote the points incident to $\widetilde{\ell_0}$ by $\widetilde{P_i}$ ($i\in [0, q]$) and the lines incident to $\widetilde{P_0}$  by $\widetilde{\ell_i}$  ($i\in [0, q]$). We call the point set $\{\widetilde{P_i} : 0<i\leq q \}$ and line set   $\{ \widetilde{\ell_i} : 0<i\leq q\}$ \textit{major point}s and \textit{major line}s in the construction, respectively. The set $\cC:=\{\widetilde{P_i} : 0<i\leq q \}\cup\{ \widetilde{\ell_i} : 0<i\leq q\}$ is the \textit{core} of the construction. Let's call the points and lines which are not in the core \textit{common points} and \textit{common lines} (altogether the \textit{common vertices}).
\end{nota}


\noindent\textbf{Construction}\\
Our partition set system $\cH$ consists of $4$ subsystems:  $$\cH=\{H_0\} \cup \cH_1 \cup \cH_2 \cup \{H_{-1}\},$$ where $\cH_1= \bigcup_{i=1}^k H_i$, $\cH_2= \bigcup_{i=k+1}^{k+l} H_i$.\\

$H_0$ is defined as
$H_0:= \{\widetilde{P_i} \mid i\in [1\ldots q]\}$, furthermore $H_{-1}$ is defined such that it completes the system, that is,  $\bigcup_{i=-1}^{k+l} H_i= \cP[G(\Pi_q)] \cup \cL[G(\Pi_q)]$ with $H_{-1}$ being disjoint from any other $H_i$.

Now we define the set system $\cH_1$.

Any $H \in \cH_1$ is built up the following way:  choose a major  point $\widetilde{P_i}$  and a major line $\widetilde{\ell_j}$  which will be the base of the set $H$. Divide the point set $\widetilde{\ell_j}\setminus \widetilde{P_0}$ into two equal parts.  Put the points of the first part into $H$ and put also every line determined by $\widetilde{P_i}$ and points of the second part. This way $|H|=q$ will hold for every element of $\cH_1$. Any subset of  $\cP \cup \cL$ that can be created this way is called \textit{a $\zeta$-set} (on the base point $\widetilde{P_i}$ and base line $\widetilde{\ell_j}$).

$k\approx 3\log_2 q$ $\zeta$-sets will be chosen randomly in such a way that almost all points and lines are uniquely determined by the distances from the sets of $\cH_1$, if we restrict ourselves to the common vertices.

Finally, $\cH_2$ will consist of $l \approx \log_2 q$ sets which will distinguish all of the remaining non-separated pairs of vertices. The existence of the these sets is implied by the following lemmas.

\begin{lemma}\label{kiszamolo} One can choose $k$ suitable $\zeta$-sets in order to separate any pair of elements from $(\cP \setminus \{\widetilde{P_i}\})$ and any pair from  $(\cL \setminus \{\widetilde{\ell_i}\})$ except for at most $m(k)$ pairs in total, with $$m(k)=\frac{2\binom{q^2}{2}}{2^k}.$$
\end{lemma}

\begin{proof}[Proof of Lemma \ref{kiszamolo}]
We choose $k$ points and $k$ lines from the major point and line set uniformly at random, and index them by $\widetilde{P_i}$ and $\widetilde{\ell_i}$, $i\in [1\ldots k]$. Next, we choose $k$ $\zeta$-sets (on the base points $\widetilde{P_i}$ and base line $\widetilde{\ell_i}$, by taking   $\lfloor\frac{q}{2}\rfloor$ points from each line $\widetilde{\ell_i}$ uniformly at random, leaving the support point intact.
This enables us to calculate the expected value of  not-separated pairs.\\
We distinguish the cases when the pair is a pair of point or pair of lines, but
note that the two cases are similar due to the symmetry of the $\zeta$-sets and the duality of the structure.

In the calculation below, we omit integer part for transparency (that is, we consider the case $q$ is even) but it is straightforward to see that the reasoning works for the $q$ odd case as well.

Let $Q$ and $Q'$ be two random points in $\cP \setminus \{\widetilde{P_i}\}$. The probability of the separation essentially depends on two factors: whether or not $QQ'$ intersects $\widetilde{\ell_0}$ in a point $\widetilde{P_t}$, $t\in [1\ldots k]$; and whether or not one of $Q$ and $Q'$ is incident to $\widetilde{\ell_t}$. These subcases provide the following for a random point pair:

\[ \mathbb{P}( Q, Q' {\mbox{ not separated by $k$ random $\zeta$-set}} )= \]
\[\mathbb{P}( Q, Q' {\mbox{ not separated by $k$ random $\zeta$-set}} \ \& \  QQ'\cap \widetilde{\ell_0} \not\in  \{\widetilde{P_t}, t\in [1\ldots k]\}  ) +  \]
\[ \mathbb{P}( Q, Q' {\mbox{ not separated by $k$ random $\zeta$-set}} \ \& \ QQ'\cap \widetilde{\ell_0} \in  \{\widetilde{P_t}, t\in [1\ldots k]\}  ) \]
 \[ \leq \frac{q-k+1}{q+1} \left(\frac{q-2}{2q-2}\right)^k + \frac{k}{q+1}\left( \frac{q-1}{q}\left(\frac{q-2}{2q-2}\right)^{k-1}      \right). \]

Indeed, suppose first that   $QQ'$ intersects $\widetilde{\ell_0}$ in a point outside $\widetilde{P_t}$, $t\in [1\ldots k]$.
That case, for every $\zeta$-set $H$ on  the base $\widetilde{P_i}$,  $\widetilde{\ell_i}$, $d(Q, H)\neq d(Q', H)$ holds if exactly one of the lines $\widetilde{P_i}Q$, $\widetilde{P_i}Q'$ belongs to $H$, hence the probability of separation by $H$ is at least $\frac{q/2}{q-1}$. (Note that equality does not hold here as $Q$ or $Q'$ might be a point of the  $\zeta$-set.)

On the other hand,  if $QQ'$ intersects $\widetilde{\ell_0}$ in a point $\widetilde{P_t}$, $t\in [1\ldots k]$, then the above argument works for all but one  $\zeta$-set, $H$ (on the base $\widetilde{P_t}$,  $\widetilde{\ell_t}$). However, if  $Q\not \in  \widetilde{\ell_t},   Q'\not \in \widetilde{\ell_t}$, then  $d(Q, H)= d(Q', H)$ surely holds, while  if  $Q \in  \widetilde{\ell_t},$ or $ Q' \in \widetilde{\ell_t}$, then  $d(Q, H)\neq d(Q', H)$ only if $Q$ or $Q'$ is a point in $H$.

Easy calculation shows that  \[ \frac{q-k+1}{q+1} \left(\frac{q-2}{2q-2}\right)^k + \frac{k}{q+1}\left( \frac{q-1}{q}\left(\frac{q-2}{2q-2}\right)^{k-1}      \right)<\left(\frac{1}{2}\right)^k. \]

Taking into consideration the number of point pairs, and the dual case for the number of line pairs, we obtain a bound on the expected value of the non-separated point pairs and line pairs:

\[\mathbb{E}(  {\mbox{ not separated pairs by $k$ random $\zeta$-set}} ) =\] \[ 2\binom{q^2}{2}[ \mathbb{P}( Q, Q' {\mbox{ not separated by $k$ random $\zeta$-set}}) ] < \frac{2\binom{q^2}{2}}{2^k}. \]

The statement thus follows.	
\end{proof}


\cut{
To state the following lemma, we need to introduce a notation.
\begin{nota} \normalfont
A set system $A_1, A_2, \ldots, A_d$ of disjoint sets  is a \textit{ system of weight $w$} if the number of pairs in the sets adds up to $w$.
\end{nota}

\begin{lemma}\label{elrendezo} Given an incident pair of point and line (as a support) and the corresponding core $\cC$ in $G(\Pi_q)$, a  system $\textbf{S}$ of disjoint sets in $V(G(\Pi_q))\setminus \cC$  of weight at most $q/2$, and $(3+o(1))\log_2 q$ $\zeta$-sets with distinct base points and lines ($\widetilde{P_i}$ $\widetilde{\ell_i}$),  one can always choose a system $\cH'$ of $(1+o(1))\log_2 q$ disjoints sets consists of points and lines  disjoint to the $\zeta$-sets and to the major points such that

\begin{itemize}
\item The elements of $\cH'$ separate the major points and also separates the major lines (defined by the support),
\item The elements of $\cH'$ separate the given pair of vertices.
\item The support line and point is separated from any prescribed set from the set system  $\textbf{S}$.
\end{itemize}
\end{lemma}}


\begin{lemma}\label{elrendezo} One can choose a system $\cH_2$ of $\lceil\log_2 q\rceil$ disjoint sets consist of both points and lines which are disjoint to the $\zeta$-sets and to the major points as well such that with these new sets the corresponding representations of the vertices will be pairwise different.
\end{lemma}

\begin{proof}
Let's look at a table of the representations of the vertices so far, with $H_{ij}$ denoting a chosen $\zeta$-set on base point $\widetilde{P_i}$ and base line $ \widetilde{\ell_j}$.

\[
\begin{array}{c|c|ccccc|c|c|ccccc|c|}
~ & \widetilde{P_0} & \widetilde{P_1} & \dots &  & \dots & \widetilde{P_q} & P_1 \dots P_{q^2} & \widetilde{\ell_0} & \widetilde{\ell_1} & \dots & \widetilde{\ell_j} & \dots & \widetilde{\ell_q} & \ell_1 \dots \ell_{q^2} \\ \hline
H_0 & 2 & 0 & \dots & 0 & \dots & 0 & 2 & 1 & 3 & \dots & 3 & \dots & 3 & 1 \\ \hline
H_{ij} \in \cH_1 & 2 & 2 & \dots2 & 1 & 2\dots & 2 & 0/1/2 & 2 & 2 & \dots2 & 1 & 2\dots & 2 & 0/1/2
\end{array}
\]

Considering this table, one can determine which pairs of vertices could have the same representation:

\begin{enumerate}
\item pairs of common vertices not being separated after Lemma \ref{kiszamolo},
\item pairs of major points and also pairs of major lines (which are not basis of $\zeta$-sets in $\cH_1$),
\item the support line and some common lines and dually the support point and some common points may form some non-separated pairs.
\end{enumerate}

These are the only possibilities which we need to take care of for the suitable choice of $\cH_2$. This observation motivates us how to choose lines and points to a member $H \in \cH_2$. In the following we will build up $\cL[H]$, and then $\cP[H]$  will be chosen in the same way dually. But before that we need some structural observations.

We make an auxiliary graph $X$ with vertex set consists of those common vertices which are in the remaining non-separated pairs and two such vertices are joined with an edge if they have the same representation so far. Clearly $X$ is just the disjoint union of some cliques. Moreover by the choice of $H_0$ in every clique either every vertex is a point of $\Pi_q$ or every vertex is a line of $\Pi_q$. Just for convenience let's assume that we choose $k=\lceil3\log_2 q\rceil  + 3$ $\zeta$-sets in the first part of the construction. That means by Lemma \ref{kiszamolo} that we have at most $\frac{q}{8}$ pairs of vertices which have not been separated yet. Hence in $X$ the number of edges is at most $\frac{q}{8}$. Together with the observation above we get that $X$ has at most $\frac{q}{4}$ vertices.

Modify $X$ a little bit by adding the support point and line to it if needed (let's call the graph we get this way $X'$), namely whenever there is a common point or line which is non-separated from the support point or line respectively. There are three options for $\widetilde{P_0}$ (and similarly for $\widetilde{\ell_0}$):

\begin{itemize}
\item $\widetilde{P_0}$ have already been separated from every other vertex in $G(\Pi_q)$.
\item $\widetilde{P_0}$ has the same representation as some of the points of $X$, hence joins that clique in $X'$.
\item $\widetilde{P_0}$ has the same representation as a common point not in $X$ (hence it's exactly one vertex), therefore they both inserted into $X'$ with an edge between them.
\end{itemize}

Either way the number of vertices in $X'$ is at most $\frac{q}{4} + 4$. Denote the vertices in $X' \cap \cP[G(\Pi_q)]$ with $\cP[X']$ and similarly $\cL[X']$ will denote the vertices in $X'$ corresponding to a line of $\Pi_q$. By the note above we know that there is no edge in $X'$ between $\cP[X']$ and $\cL[X']$.


Now let's recall the notion of a \textit{searching set} for a search problem. Roughly speaking we would like to distinguish all elements of a set by pointing out some subsets and create a 0-1 vector for every element of the set where the value of the $j^{\mathrm{th}}$ coordinate is 1 iff this element is inside the $j^{\mathrm{th}}$ chosen subset. These distinguishing subsets are also called searching sets. It is known that for an $n$ element set we need $\lceil \log_2 n \rceil$ such searching set to reach our goal by distinguishing all of the elements of the set.

The main idea behind the selection of $\cL[H]$ is based on these searching sets. Namely we will consider a family of searching sets on the $q$ major points and also on $\cP[X']$ and on $\cL[X']$ too, with sizes half of their corresponding domain set. For an $H \in \cH_2$ we will choose one of the searching sets for the major points, for $\cP[X']$ and also for $\cL[X']$. Let's denote these searching sets with $T(H)$, $Q(H)$ and $R(H)$ respectively. We can assume that $\widetilde{P_0}$ and $\widetilde{\ell_0}$ are always outside of every chosen searching sets without the loss of generality.

We are going to prove that we can find $\frac{q}{2}$ such common lines (they will form $\cL[H]$) which satisfies the following properties:

\begin{itemize}
\item through every point in $T(H)$ there is exactly one line from $\cL[H]$,
\item through every point in $Q(H)$ there is exactly one line from $\cL[H]$,
\item they are not going through any points of $\cP[X'] \setminus Q(H)=:Q(H)^{C}$,
\item they are not in $\cL[X'] \setminus R(H)=:R(H)^{C}$,
\item they have not been assigned to any sets in our construction yet.
\end{itemize}


Let's call those common lines which satisfy the last three requirements \textit{free} lines. Let's make another auxiliary graph denoted by $Y$ which is a bipartite graph where the first class consists of the vertices of $T(H)$ and the other class is just $\cP[G(\Pi_q)] \setminus \left ( \cup_{i=0}^{q+1} \widetilde{\cP_i} \right ) \setminus Q(H)^C$ and there is an edge between two vertices (obviously from different classes) iff the line defined by these two vertices is a free line. We need to give a lower bound on the number of free lines through an arbitrary point from $T(H)$ and also from $Q$ because of the first two requirements.

Consider a point $v \in T(H)$. Through this point there are $q$ common lines but it may happen that this point was chosen as a base to a $\zeta$-set therefore it is possible that $\frac{q}{2}$ of these common lines were used before. Since we are not going to use lines through the points of $Q(H)^C$ and lines from $R(H)^C$ it could rule out another $|Q(H)^C| + |R(H)^C| = \frac{1}{2} \left ( \frac{q}{4} + 4 \right ) = \frac{q}{8} + 2$ lines through $v$. Furthermore if $v$ was in some searching sets on the major points before during this phase then in every such case there was one line which was inserted into that particular set from $\cH_2$. This gives us a lower bound on the number of free lines through an arbitrary point of $T(H)$:

\[
|\{\mathrm{free~ lines~ through~ }v \in T(H)\}| \ge q - \frac{q}{2} - \left ( \frac{q}{8}+2 \right ) - \log_2 q = \frac{3q}{8} - \log_2 q - 2
\]

Note that the degree of $v$ in $Y$ is at least $q \cdot \left (\frac{3q}{8} - \log_2 q - 2 \right )$ because every free line through $v$ has $q$ points on it which corresponds to $q$ edges in the graph. Consider now an arbitrary point $u \in Q(H)$. Through $u$ there are $\frac{q}{2}$ common lines which meets the support line in $T(H)$. Again we don't want to use lines through the points of $Q(H)^C$ and lines from $R(H)^C$ which could rule out $\frac{q}{8} + 2$ lines just as above. And if $u$ was in some searching sets on $\cP[X']$ before during this phase then in every such case there was one line which was inserted into that particular set from $\cH_2$. Hence:

\[
|\{ \mathrm{free~ lines~ through~ }u \in Q(H)\}| \ge \frac{q}{2} - \left ( \frac{q}{8}+2 \right ) - \log_2 q = \frac{3q}{8} - \log_2 q - 2
\]

Now we can see that if $q$ is large enough then there are many free lines through these points in $Y$. Moreover for an arbitrary point $u\in Q(H)$ if we ignore those lines which contains another point of $Q(H)$ then there is at least $\frac{3q}{8} - \log_2 q - 2 - \left (\frac{q}{8}+2\right ) = \frac{q}{4} - \log_2 q -4$ such free lines through $u$. Let's consider the points of $Q(H)$ one by one then we can choose an edge (thus a free line) for the first one which fulfills the requirements and does not contain any other points from $Q(H)$. Drop the meeting point of this line and the support line for any other points of $Q(H)$ because of the first requirement. Then we can continue this in a greedy way because of the counting above (for the last member of $Q(H)$ we drop another at most $\frac{q}{8}+2$ meeting points but the number of free lines through that point is still at least $\frac{q}{4} - \log_2 q - 4 - \left (\frac{q}{8}+2\right ) = \frac{q}{8} - \log_2 q - 6$).

Now we just need to choose one line through every uncovered points of $T(H)$ carefully. Note that if we drop those lines through these points which meet $Q(H)$ (and in parallel delete the $q$ edges for each of them from $Y$) then again there remains at least $\frac{3q}{8} - \log_2 q - 2 - \left (\frac{q}{8}+2\right ) = \frac{q}{4} - \log_2 q -4$ free lines through them. By choosing from the free lines greedily works again because of the calculations above. Dually one can construct $\cP[H]$ in a similar way which completes the set $H \in \cH_2$. By repeating this argument we can construct $\lceil\log_2 q\rceil$ such sets in $\cH_2$ (we included the decreasing of the degrees above therefore this greedy approach will work).
\smallskip

In the preceding paragraphs we just showed a way of choosing $\log_2 q$ sets all of which fulfills the requirements for it's lines and dually for it's points too. The only thing remained is to verify that these set in $\cH_2$ really take care of every pair of non-separated vertices after choosing $H_0$ and $\cH_1$.


The major points (and dually the major lines) are indeed separated because for a particular $H \in \cH_2$ those major points will have coordinate 1 which are inside $T(H)$ since we chose a line going through them and the other major points have coordinate 2. Since we chose a family of searching sets on these major points then eventually they will be distinguished at the end. If we examine all of the other remaining pairs we can notice that the coordinates for $\widetilde{P_0}$ and $\widetilde{\ell_0}$ will be 2 for every set in $\cH_2$ and for a particular $H \in \cH_2$ the points of $Q(H)$ will have coordinate at most 1 (we chose a line through them and maybe we put them into $\cP[H]$) but for the points of $Q(H)^C$ the coordinate is surely 2 (did not choose a line through them and we exclude them from being in $\cP[H]$). Similar argument hold for the lines of $R(H)$ and $R(H)^C$. Again since we chose a family of searching sets on the $\cP[X']$ and on $\cL[X']$ all of the remaining pairs will be separated too.
\end{proof}

\bigskip

Observe that by adding all of the non-used vertices of $G(\Pi_q)$ to $H_{-1}$ we obtain a resolving partition indeed,
with  $1+(3\lceil\log_2 q\rceil + 3)+\lceil\log_2 q\rceil + 1 \leq 4\lceil\log_2 q\rceil + 5$ classes which completes the proof of Theorem \ref{felso}.

\begin{remark}
One can easily see that in these Lemmas we don't rely heavily on the parity of $q$, if $q$ is odd everything still works with slight modifications.
\end{remark}

\section{Further related problems and remarks}

Although the lower and upper bounds we proved do not match, we strongly believe that the construction given in the upper bound is optimal in some sense. The reason for this is the following: one have to create a set system where the majority of the sets  contain neither more points  nor more lines than $cq$ for a small constant $c$. Indeed, the result of Blokhuis \cite{blok} implies that $cq$ lines are incident with at least roughly $\frac{c}{c+1}q^2$ points, hence a set containing this many lines assign the same distance for the majority of the points. \\
This observation in fact improves our lower bound via the result of Katona \cite{Katona} on separating systems of given size, but only in the remainder term.
Thus several questions remain open concerning the optimal construction.

\begin{problem} Can the above bounds be improved if the plane in coordinatized, that is, $G$ is the incidence graph of the plane $\PG(2,q)$?
\end{problem}

\begin{problem} Prove that there exists a constant $c$ for which $\pd(G(\Pi_q))=(c+o(1))\log_2q$.
\end{problem}

Our construction mainly based on sets of collinear points and lines forming a pencil, in order to separate approximately half of the lines (incident to the points) from the other lines and do the same for the points for every set in the resolving partition. Note that the points of a maximal arc and the dual configuration owns the same property. This motivates the following natural question.

\begin{problem} Does there exist a set of disjoint ovals in $\Pi_q$ of cardinality $c\log_2q$ ($c\leq 4$) which separates the lines of the projective plane $\Pi_q$?
\end{problem}

A related problem only requires a set of ovals to cover (intersect) every line.

\begin{problem} What is the minimal cardinality of a set of (disjoint) ovals   in $\Pi_q$ for which no line is skew to all of them?
\end{problem}

It is believed \cite{Szonyi2, Ughi} that the order of magnitude is $O(\log q)$ for $q$ odd, which provides $O(q\log q)$ points on the plane, even the size of small minimal blocking (point)sets is much less. Note that the $q$ even case is completely different, where $3$ ovals can cover every line in the Galois plane  $\PG(2,q)$ due to Sz\H{o}nyi, Ill\'es and Wettl \cite{Szonyi2}.

Our result can also be considered as a first step to the determination of the partition dimension of incidence graph of symmetric structures in general, analogously to the metric dimension case \cite{Bailey}.

\end{document}